\documentclass[10pt]{amsart}
\usepackage{amssymb, amsmath}
\usepackage{graphics,xcolor}
\usepackage{latexsym,amsmath,amssymb,amsthm,amsfonts,mathrsfs}
\usepackage{amscd}
\usepackage[arrow, matrix, curve]{xy}
\usepackage{tikz-cd}

\DeclareMathOperator{\gal}{Gal}

\DeclareMathOperator{\pic}{Pic}
\DeclareMathOperator{\nm}{Nm}

\DeclareMathOperator{\Sp}{Sp}
\DeclareMathOperator{\gsp}{Gsp} 
\DeclareMathOperator{\im}{Im}

\theoremstyle{plain}
\newtheorem{thm}{Theorem}[section]
\newtheorem{theorem}[thm]{Theorem}
\newtheorem{lemma}[thm]{Lemma}
\newtheorem{corollary}[thm]{Corollary}

\theoremstyle{definition}

\newtheorem{remark}[thm]{Remark}

\newtheorem{definition}[thm]{Definition}

\numberwithin{equation}{thm}

\newcommand{\sC}{{\mathcal C}}
\newcommand{\sD}{{\mathcal D}}
\newcommand{\sE}{{\mathcal E}}

\newcommand{\sL}{{\mathcal L}}

\newcommand{\sO}{{\mathcal O}}
\newcommand{\sP}{{\mathcal P}}

\newcommand{\sU}{{\mathcal U}}

\newcommand{\C}{{\mathbb C}}
\newcommand{\D}{{\mathbb D}}

\renewcommand{\H}{{\mathbb H}}

\renewcommand{\P}{{\mathbb P}}
\newcommand{\Q}{{\mathbb Q}}
\newcommand{\R}{{\mathbb R}}

\newcommand{\V}{{\mathbb V}}

\newcommand{\Z}{{\mathbb Z}}

\newcommand{\rk}{{\rm rank}}

\newcommand{\Hom}{{\rm Hom}}

\title[]{On the Existence of Shimura curves in the Prym locus of abelian covers of projective line}
\author{Abolfazl Mohajer}
\email{abmohajer83@gmail.com}

\begin{document}
\begin{abstract}
Using the theory of Higgs bundles and their stabitlity properties associated to fibered surfaces and the Viehweg-Zuo characterization of Shimura curves in the moduli space of abelian varieties in terms of Higgs bundles, we prove that there does not exist any non-compact Shimura curves in the Prym locus of totally ramified $\Z_{2p}$- or $\Z_{2p}\times (\Z_{p})^{m-1}$-covers of the projective line in $A_{g}$ for $g\geq 8$, where $p\geq  5$ is a prime number. 
\end{abstract}
	
\maketitle
	
\section{Introduction}
In this paper we prove non-existence results for non-compact Shimura curves contained generically in the Prym locus of abelian covers of the projective line. A list of such Shimura curves was found in the paper \cite{CFGP} by E. Colomobo, P. Frediani, A. Ghigi and M. Penegin. Further interesting results in this direction can be found also in the paper \cite{CF} by E. Colomobo and P. Frediani. Let us explain the general framework of the problems that we consider: Let $C$ be a smooth projective curve of genus $g$ and $\eta\in \pic^0(C)$ an element of order two in the Picard group  i.e., $\eta\neq \mathcal{O}_C$ but $\eta^2=\mathcal{O}_C$. This element, viewed as a line bundle or a divisor on $C$ yields an  \'etale double cover $h:\tilde{C}\to C$. Using the morphism $h$, we obatin the so-called \emph{norm map} $\nm:\pic^0(\tilde{C})\to \pic^0(C)$ given by $\sum r_ix_i\to \sum r_ih(x_i)$. Consider the isomorphism class $[C,\eta]$ of the pair $(C,\eta)$. There is an abelian variety of dimension $g-1$ associated to this isomorphism class. It is defined to be the connected component of the kernel $\ker\nm$ containing the identity. We denote this abelian varitey by $P(C,\eta)$ or $P(\tilde{C}, C)$ and call it the \emph{Prym variety} associated to the class $[C,\eta]$ or to the double cover $h:\tilde{C}\to C$. There is a moduli space $R_g$ for the isomorphism classes $[C,\eta]$. In this paper (as in \cite{CFGP}) we also consider the case where the double cover $h$ is ramified above exactly two poins. In this case, there is a divisor $B$ of degree 2 on the smooth curve $C$ that is in the linear series $|\eta^2|$ and we consider the isomorphism classes $[C,B,\eta]$ which are parametrized by a scheme denoted by $R_{g,2}$. We can define the Prym variety $P(\tilde{C}, C)$ in the analogous way as above also in this case. It is an abelian variety of dimension $g$. Let $A_g$ be the moduli space of complex principally polarized abelian varieties. In the two cases above, the Prym variety is principally polarized and there is a map $\mathscr{P}:R_{g}\to A_{g-1}$ (resp. $\mathscr{P}:R_{g,2}\to A_g$) which assings to a pair $[C,\eta]$ (resp. $[C,B,\eta]$) the Prym variety $P(\tilde{C}, C)$. We call this map the \emph{unramified}(resp. \emph{ramified}) \emph{Prym map}. In both of the above cases the paper \cite{CFGP} gives various examples of one-parameter families $(C_t,\eta_t)$ ($t\in T=\P^1\setminus\{0,1,\infty\}$) for which the image $\mathscr{P}(T)$ is a Shimura curve in $A_{g-1}$ or $A_{g}$. The families that the authors consider are families of Galois coverings $C_t\to\mathbb{P}^{1}$. Note that, following \cite{CFGP}, we denote the image of  $T$ in $ R_g$ (resp. $R_{g,2}$) via the assignment $t\mapsto [C_t,\eta_t]$ (resp. $[C,B,\eta]$) also by $T$ see also section 2.

More precisely, the families that authors consider are of the following form: consider families of Galois covers $\tilde{C_t}\to \mathbb{P}^{1}$ with abelian Galois group $\tilde{G}$ and an  involution $\sigma$ such that the double covering $\tilde{C_t}\to \tilde{C_t}/\langle \sigma\rangle$ is either \'etale or branched above two (distinct) points. The Galois covering $\tilde{C_t}\to \mathbb{P}^{1}$ is determined by an epimorphism of groups, where $s$ is the number of branch points $\tilde{\theta}:\Gamma_s\to\tilde{G}$ with branch points $t_1,\dots,t_s\in \mathbb{P}^{1}$. Indeed the group $\Gamma_s$ is isomorphic to the fundamental group of $\mathbb{P}^{1}\setminus\{t_1,\dots,t_s\}$ and hence has the following presentation.
\begin{equation}\label{presentation}
\Gamma_s=\langle\gamma_1, \ldots, \gamma_s|\gamma_1\cdots\gamma_s=1\rangle
\end{equation}


 In \cite{CFGP} one then finds 52 examples of families $R(\tilde{G},\tilde{\theta},\sigma)$ such that the Zariski closure $\overline{\mathscr{P}(T)}$ of the image of Prym map is a Shimura curve in $A_{g-1}$ (resp. $A_{g}$).\\

Our main objects of interest in this article are families  $f:\tilde{\sC}\to T$ of abelian coverings of $\mathbb{P}^{1}$ which are obtained by varying the branch points of the covers (for a fixed matrix $A$ as mentioned above). Hence we get a family in $R_g$ (resp. $R_{g,2}$) which is indeed the image of $T$ in $R_g$ (resp. $R_{g,2}$) and we shall denote it again by $T$ in the above.  We consider the general case where the variety $T$ is of dimension $s-3$, so in this paper, unlike \cite{CFGP}, we do not assume that $T$ is 1-dimensional or in other words that $s=4$. Note that for each fiber, which is itself a covering $\tilde{C_t}\to \mathbb{P}^{1}$, we obtain the following diagramm by taking the quoteint by the action of $\sigma$.

\[\begin{tikzcd}[row sep=4em, column sep=2em]
    \tilde{C_t} \arrow[rr, "\pi_t"] \arrow[swap, dr] & & C_t=\tilde{C_t}/\langle\sigma\rangle \arrow[dl] \\
    & \P^1\cong\tilde{C_t}/\tilde{G}\cong C_t/G \\[-4.6em]
\end{tikzcd}\]

In this paper, we study Shimura curves in the Prym locus of abelian $\tilde{G}$-covers of $\P^1$. We then show that for $\Z_{2p}$- or $\Z_{2p}\times (\Z_{p})^{m-1}$-covers of the projective line, the Zariski closure $Z=\overline{\mathscr{P}(T)}$ does not contain Shimura curves in $A_g$ if $g>8$. We remark that in \cite{CF} upper bounds for the dimension of a germ of a totally geodesic submanifold, and hence for the dimension of a special subvariety in the Prym locus are proven.\\

 In section $2$, we construct abelian covers of $\mathbb{P}^{1}$ with a method different than the one in \cite{CFGP}. We also construct Prym varieties and  Prym maps of the families that we consider in which we follow \cite{CFGP}. Fix an $m \times s$ matrix $A$ with entries in $(\mathbb{Z}/N\mathbb{Z})$ (for fixed $m$ and $s$). To each $s$-tuple $t=(z_{1},\dots,z_{s})\in (\mathbb{A}^{1}_{\mathbb{C}})^{s}$ there corresponds a Galois covering $\tilde{C_t}\rightarrow \mathbb{P}^{1}$ branched at the points $z_{j}$ with local monodromies given by the matrix $A$ and an abelian Galois group $\tilde{G}$  isomorphic to the column span of the matrix $A$. The group  $\tilde{G}$ is therefore a subgroup of $(\mathbb{Z}/N\mathbb{Z})^{m}$. The group action induces a corresponding action on the cohomology of the fibers. Using our description of abelian coverings of projective line, we can concretely compute the $\tilde{G}$ eigenspaces and therefore the eigenspaces of the $\sigma$-action which are crucial for our later computations related to Higgs bundles. \\

In paper \cite{Moh21}, the author has considered the same problem for higher dimensional Shimura subvarities of the Prym locus. The methods that we used in that paper, enabled us to prove that when $s$ is large enough, then the image of $T$ is not a special subvariety. However, the methods in \cite{Moh21} cannot exclude Shimura curves, i.e., one-dimensional Shimura subvarieties which are much harder to deal with. In this paper, we use the methods of \emph{Higgs bundles} to achieve this. In Section \ref{Shimura&Higgs}, the Higgs bundles are introduced and the Higgs bundle associated to a family of Prym varieties are studied.  The relation of Higgs bundles and Shimura curves are also explained in this section and we develop criterions for a family of covers of $\P^1$ to generate Shimura families in the Prym locus $\sP_g$. These condition will subsequently used in Section \ref{Shimura in Prym} to exclude Shimura curves in the Prym locus of cyclic and abelian covers of $\P^1$. The cyclic $\Z_p$-covers of $\P^1$ are \emph{totally ramified}, namely above each branch point there is a single ramification point. The abelian covers of $\P^1$ that we consider are totally ramified in the sense that in the matrix of the covering (see section \ref{prelim} below for the construction of abelian covers) each column has exactly one 1 entry and all other entries are 0.

\section{preliminaries on abelian covers of $\mathbb{P}^{1}$ and the Prym map} \label{prelim}
The notation and treatment of this section comes from \cite{CFGP} and also \cite{MZ} which follows that of \cite{W}. For a comprehensive explanation of construction of Prym varieties and the Prym map we refer the reader to \cite{BL}. \\

An abelian Galois cover of $\mathbb{P}^{1}$ is determined by a collection of equations as follows:\\

Fix inetegers $m$ and $s$ and let $A=(r_{ij})$ be an $m\times s$ matrix whose entries $r_{ij}$ are in $\mathbb{Z}/N\mathbb{Z}$ for some $N\geq 2$. Let $\overline{\mathbb{C}(z)}$ be the algebraic closure of $\mathbb{C}(z)$. For each $i=1,...,m,$ choose a function $w_{i}\in\overline{\mathbb{C}(z)}$ with
\begin{equation}\label{abelian equation}
w_{i}^{N}=\prod_{j=1}^{s}(z-z_{j})^{\widetilde{r}_{ij}}\text{ for }i=1,\dots, m,
\end{equation}
in $\mathbb{C}(z)[w_{1},\dots,w_{m}]$. We have denoted the lift of $r_{ij}$ to $\Z\cap [0,N)$ by $\widetilde{r}_{ij}$ and $z_j\in \C$ for $j=1,2,\dots, s$. The equations \ref{abelian equation} define only a singular affine curve in general and we consider a smooth projective model associated to this affine curve. The matrix $A$ should have the following property: the sum of the columns of $A$ is zero (when considered as a vector in $(\mathbb{Z}/N\mathbb{Z})^m$). This condition ensures that the curve defined by \ref{abelian equation} is \emph{not} branched over the infinity. The matrix $A$ is called the matrix of the covering.

The local monodromy around the branch point $z_{j}$ is given by the column vector $(r_{1j},\dots, r_{mj})^{t}$ and therefore the order of ramification over $z_{j}$ is equal to $\frac{N}{\gcd(N,\widetilde{r}_{1j},\dots,\widetilde{r}_{mj})}$. This together with the Riemann-Hurwitz formula yield the following theorem to compute the genus $g$ of the cover:
\begin{theorem}
Let $\tilde{C}\to \mathbb{P}^{1}$ be an abelian $\tilde{G}$-cover as above. Then the group $\tilde{G}$ acts on the space $H^0(\tilde{C},K_{\tilde{C}})$ and the dimension of the eigenspace $H^0(\tilde{C},K_{\tilde{C}})_n$ related to $n\in\tilde{G}$ is given by 
\begin{equation}\label{eigspace}
d_{n}=\dim H^0(\tilde{C},K_{\tilde{C}})_n=-1+\sum_{1}^{s}\langle - \frac{\alpha_{j}}{N}\rangle
\end{equation}
Also the genus of the cover is given by

\begin{equation}\label{genus formula}
\tilde{g}= 1+ d(\frac{s-2}{2}-\frac{1}{2N}\sum_{j=1}^{s}\gcd(N,\widetilde{r}_{1j},\dots,\widetilde{r}_{mj})),
\end{equation}
Here $d$ is the degree of the covering which is equal, as pointed out above, to the column span (equivalently row span)  of the matrix $A$. In this way, the Galois group $\tilde{G}$ of the covering will be a subgroup of $(\mathbb{Z}/N\mathbb{Z})^{m}$. Note also that this group is isomorphic to the column span of the above matrix. 
\end{theorem}
The following corollary is useful when we consider cyclic covers of $\P^1$ and in particular the totally ramified ones.

\begin{corollary}\label{cyc tot}
For a cyclic cover $\tilde{C}\to \mathbb{P}^{1}$ with Galois group $\Z_n$, we have
\begin{equation}\label{eigspace cyc}
d_{n}=\dim H^0(\tilde{C},K_{\tilde{C}})_n=-1+\sum_{1}^{s}\langle - \frac{nr_j}{N}\rangle
\end{equation}
In particular, if the cyclic cover is totally ramified, then
\begin{equation}\label{eig tot}
d_n=-1+s(1-\frac{n}{N})
\end{equation}
and also
\begin{equation}\label{genus tot}
\tilde{g}=(-1+\frac{s}{2})(N-1)
\end{equation}
\end{corollary}
In what follows we describe the construction of families of abelian covers of the projective line. \par Let $T \subset (\mathbb{A}_{\C}^{1})^{s}$ be the complement of the big diagonals, i.e., 
\begin{equation}\label{parameter space}
T=\mathcal{P}_{s}= \{(z_{1},\dots,z_{s})\in(\mathbb{A}_{\C}^{1})^{s}\mid z_{i}\neq z_{j} \forall i\neq j \}
\end{equation}
Fix an $m\times s$ matrix $A$ as in the beginning of this section. For each $(z_{1},\dots,z_{s})$ of $T$, we obtain an abelian cover of $\mathbb{P}^{1}$ by the using equation \ref{abelian equation}. The abelian cover has branch points $z_{1},\dots,z_{s}$ and monodromy data given by the matrix $A$. In this way we obtain a family $f:\tilde{\sC}\to T$ of smooth projective curves whose fibers $\tilde{C}_t$ are abelian covers of $\mathbb{P}^{1}$ introduced above.\par Next we will explain the construction of the Prym locus for families of abelian covers of the projective line. For the general case we refer to \cite{CFGP}. The following definition will be fundamental in the sequel.
\begin{definition}\label{Prym datum}
Fix a matrix $A$ as in the beginning of this section. An \emph{(abelian) Prym datum} (see also \cite{CFGP}, Definition 3.1) is a triple $(\tilde{G}, \tilde{\theta},\sigma)$ consisting of an abelian group $\tilde{G}$ generated by the columns of the matrix $A$ and $\sigma\in \tilde{G}$, a 2-torsion element that is not contained in $\displaystyle \cup_i <T_i>$ where $T_i$ is the $i$-th column of the matrix  $A$ and also an epimorphism $\tilde{\theta}: \Gamma_s\to \tilde{G}$ .
\end{definition} 
Let $f:\tilde{\sC}\to T$ be a family as above whose fibers $\tilde{C}_t$ are abelian $\tilde{G}$-Galois coverings of $\mathbb{P}^{1}$ for a fixed finite abelian group Let $\tilde{G}$. Let $\tilde{C}_t\to \mathbb{P}^{1}$ be a fiber of this family and let $\sigma\in \tilde{G}$ be an involution. Set $G=\tilde{G}/\langle \sigma\rangle$ and $C_t=\tilde{C}_t/\langle \sigma\rangle$. Suppose that $\sigma$ has no fixed points so that the quotient map $\pi_t: \tilde{C}_t\to C_t$ is an \'etale double cover. This double cover is determined by a 2-torsion line bundle $\eta_t\in\pic^0(C_t)$. Let $V_t=H^0(\tilde{C}_t,K_{\tilde{C}_t})$ and let $V_t=V_{+,t}\oplus V_{-,t}$ be the eigenspace decomposition for the action of $\sigma$. We also have the  Hodge decomposition $H^1(\tilde{C_t},\mathbb{C})_-=V_{-,t}\oplus \overline{V}_{-,t}$. Set $\Lambda_t=H_1(\tilde{C_t},\mathbb{Z})_-$. We define the Prym variety (associated to the double cover $\pi_t$, or line bundle $\eta_t$) to be the following abelian variety
\begin{equation}\label{defPrym}
P(C_t,\eta_t)=V^{*}_{-,t}/\Lambda_t,
\end{equation}
which is an abelian variety of dimension $g-1$. For a more detailed treatment of Prym varieties, see \cite{BL}. \par In analogy with the above, we can define the moduli space of ramified double coverings and the corresponding Prym varieties. Suppose that the double covering $\pi:\tilde{C}\to C$ is branched over a divisor $B$ of degree 2. The covering is related to a line bundle $\eta$ is of degree $1$ on $C$ and $B$ is a divisor in the linear series $|\eta^2|$. We denote the moduli space of pairs $[C,B,\eta]$ by $R_{g,2}$. For such an isomorphism class, or the associated double cover, we define the Prym variety $P(C_t,\eta_t)$ (or $P(\tilde{C}, C)$) exactly by \ref{defPrym}. We therefore obtain the Prym map $\mathscr{P}:R_{g,2}\to A_g$ which associates to $[C,B,\eta]$ the Prym variety $P(\tilde{C}, C)$.

\par Let us say some words about the moduli space $A_{g}$ of complex principally polarized abelian varieties of dimension $g$ which is the most important moduli space throughout this note. We have $A_{g}=\Gamma_g(N)\setminus\H_g$, where 

$\H_g:=\{M\in M_g(\C)\mid M^t=M, \im M\geq 0\}$ is the \emph{Siegel upper half space  of genus $g$} and $\Gamma_g(N)=\{g\in \Sp_{2g}(\Z)|g\equiv I_N\text{ mod }N\}, N\geq 3$. By this construction, $A_{g}$ has the structure of a Shimura variety.

We also have that $\H_g=\gsp_{2g}(\R)/K$, where $\gsp_{2g}$ is the $\Q$-group of symplectic similitudes of the $2g$-dimensional standard symplectic $\Q$-vector space $\Q^{2g}$ and $K$ is a maximal compact subgroup.

 An algebraic subvariety of the form $Y=\Gamma_1\setminus\D\hookrightarrow A_{g}$ induced by an injective homomorphism $\D\hookrightarrow \Sp_{2g}(\R)$ of algebraic groups is called a \emph{special (or Shimura) subvariety} of $A_{g}$. A special subvariety is thus a totally geodesic subvariety.\par There is a map $T\to R_g$ which is given by $t\mapsto (C_t,\eta_t)$ and the image of $T$ under this map (which is again denoted by $T$) is a variety of dimension $s-3$, see \cite{CFGP}, p. 6. 

Therefore the above family gives rise to a subvariety of $R_g$ of dimension $s-3$. In this paper we are interested in determining whether the subvariety $Z=\overline{\mathscr{P}(T)}\subset A_{g-1}$ contains a Shimura curve. We have the analogous construction for ramified Prym map and we determine whether the closure $Z=\overline{\mathscr{P}(T)}\subset A_{g}$ for $T\hookrightarrow R_{g,2}$(which is of dimension $s-3$) contains a Shimura curve. 
	
\subsection{Abelian covers and their eigenspaces}
We begin this subsection by the following remark
\begin{remark} \label{abcharac}
Let $G$ be a finite abelian group. We call the group $\mu_G=\Hom(G,\mathbb{C}^{*})$ the \emph{character group} of $G$. The group $\mu_G$ is isomorphic to $G$. However the isomorphism is non-canonial. Let us therefore describe the isomorphism that we choose throughout this note: We first suppose that $G=\mathbb{Z}/N$ is a cyclic group and consider the isomorphism between $\mathbb{Z}/N$ and the group of $N$-th roots of unity in $\mathbb{C}^{*}$ given by $1\mapsto \exp(2\pi i/N)$. Now the group $\mu_G$ is isomorphic to the group of $N$-th roots of unity by $\chi\mapsto \chi(1)$. In general $G$ is a product of finite cyclic groups and we extend this isomorphism to an isomorphism $\varphi_G: G \xrightarrow{\sim} \mu_G$. Throughout the paper, we use this isomorphism to identify the elements of $G$ with its characters.
\end{remark}
Let $n=(n_1,\dots, n_m)\in \tilde{G}\subseteq (\mathbb{Z}/N)^{m}$ which can also be viewed as a character by Remark \ref{abcharac}.  In particular $n$ can be regarded as an $1\times m$ matrix and we consider the matrix product $n\cdotp A$ and set $n\cdotp A=(\alpha_{1},\dots,\alpha_{s})$. Note that the operations here are all done in $\mathbb{Z}/N$ but the $\alpha_{j}$ are considered as integers in $[0,N)$.\par Let $H^0(\tilde{C},K_{\tilde{C}})_n$ be the eigenspace of $\tilde{G}$-action with respect to $n$. This is a vector space over $\mathbb{C}$.  It is shown in \cite{MZ}, proof of Lemma 5.1, that a basis for this $\mathbb{C}$-vector space is given by 1-forms $\omega_{n,\nu}=z^{\nu} w_{1}^{n_1}\cdots w_{m}^{n_m}\displaystyle \prod_{j=1}^{s} (z-z_j)^{\lfloor -\frac{\tilde{\alpha_j}}{N}\rfloor}dz$. Here $\tilde{\alpha_j}$ is the lift of $\alpha_j$ to $[0,N)$ and $0\leq\nu\leq d_{n}-1=-2+\displaystyle \sum_{j=1}^{s}\langle-\frac{\alpha_{j}}{N}\rangle$. A different method for computing the dimension of the above vector space can be found in \cite{MZ}, Proposition 2.8.\par The action of $\sigma$ on 1-forms on an abelian cover of $\mathbb{P}^1$ is given by $w_i\mapsto -w_i$ for some subset of $\{1,\dots, m\}$
(and also as identity on $w_j$ for $j$ in the complement of this subset).  In this paper however we always assume that the abelian covers are of the form $\mathbb{Z}_{2n}\times(\mathbb{Z}_{n})^{m-1}$  and the action of $\sigma$ is given by $\sigma(w_1)= -w_1$ and $\sigma(w_i)= w_i$ for $i\in\{2,\dots,m\}$.  Then we have that $G=\tilde{G}/\langle\sigma\rangle\simeq (\mathbb{Z}_{n})^{m}$.
\begin{lemma} \label{eigspacedim}
The group $\tilde{G}$ acts on the spaces $H^0(\tilde{C},K_{\tilde{C}})_{-}$ and for $n=(n_1,\dots,n_m)\in \tilde{G}$, it holds that $H^0(\tilde{C},K_{\tilde{C}})_{-,n}=H^0(\tilde{C},K_{\tilde{C}})_{n}$ if and only if $n_1$ is odd and $H^0(\tilde{C},K_{\tilde{C}})_{-,n}=0$ otherwise. Similar statements hold for $H^1(\tilde{C},\C)_{-,n}$.
\end{lemma}
\begin{proof}
By the above description of the space $H^0(\tilde{C},K_{\tilde{C}})$ and also the action of $\sigma$, the $\sigma$-eigenspace $H^0(\tilde{C},K_{\tilde{C}})_+$ is the set of all $\omega_{n,\nu}$ with $n_1$ even and $H^0(\tilde{C},K_{\tilde{C}})_-$ is the set of all $\omega_{n,\nu}$ with $n_1$ odd. If $n\in \tilde{G}$ (or by Remark \ref{abcharac}, the corresponding character) the eigenspace $H^0(\tilde{C},K_{\tilde{C}})_{-,n}$ is given by:
\begin{align}
H^0(\tilde{C},K_{\tilde{C}})_{-,n}=\begin{cases}
0, & \text{if } 2|n_1,\\
H^0(\tilde{C},K_{\tilde{C}})_{n} & \text{otherwise.} 
\end{cases} 
\end{align}
and in general we have:
\begin{equation}
H^1(\tilde{C},\C)_{-,n}=\begin{cases}
0, & \text{if } 2|n_1,\\
H^1(\tilde{C},\C)_{n} & \text{otherwise.} 
\end{cases} 
\end{equation}
\end{proof}
Recall that $G=\tilde{G}/\langle \sigma\rangle$. In the sequel, we will need non-zero eigenspaces. Therefore, in view of Lemma \ref{eigspacedim}, it is more convenient to consider the $G$-action of the eigenspaces. We have the following lemma in the cyclic case
 \begin{lemma} \label{dimeigspaceG}
The group $G=\Z_n$ acts on the spaces $H^0(\tilde{C},K_{\tilde{C}})_{-}$ it holds that $H^0(\tilde{C},K_{\tilde{C}})_{-,i}=H^0(\tilde{C},K_{\tilde{C}})_{\overline{i}}$, where $\overline{i}$ is the representative of $i$ modulo $n$.
\end{lemma}

\section{Shimura curves and Higgs bundles}\label{Shimura&Higgs}
\subsection{Shimura varieties}
In this subsection, we will breifly introduce and review some facts about Shimura varieties and special subvarieties. In this note we only consider connected Shimura varieties which are complex algebraic varieties of the form $\Gamma\backslash X^+$, where $X^+$ is a Hermitian symmetric domain and $\Gamma$ is a congruence subgroup of a semisimple algebraic group $G^{der}(\Q)$ acting on $X^+$. Inside $\Gamma\backslash X^+$ there are Shimura or special subvarieties associated with Shimura subdata, see \cite{Mil}, \S 5. In simple terms, they are subvarieties of the form $\Gamma\backslash {X^{\prime}}^+\subseteq \Gamma\backslash X^+$, where ${X^{\prime}}^+\subseteq X^+$ are equivariant embeddings of Hermitian symmetric subdomains of $X^+$. Equivariantly embedded here means that it is defined by some semi-simple Lie subgroup $G^{\prime}$ of $G^{der}(\R)$ and the inclusion ${X^{\prime}}^+\hookrightarrow X^+$ is equivariant with respect to $G^{\prime}\hookrightarrow G^{der}(\R)$. The zero-dimensional special subvarieties are CM points. In this paper, we are interested about higher dimensional special subvarieties. The most important Shimura variety for us is the moduli space of complex $g$-dimensional principally polarized abelian varieties $A_g$. Here $A_g=A_{g,l}=\Gamma_g(l)\backslash \H_g$, where $\H_g$ is the Siegel uper half-space of genus $g$ and $\Gamma_g(l)$ is the principal congruence subgroup of level-$l$ in $Sp_{2g}(\Z)$ which is the kernel of the natural map $Sp_{2g}(\Z)\to Sp_{2g}(\Z/l)$, where $l\geq 3$ is an odd integer. In this case $\Gamma_g(l)$ is torsion-free and the quotient $\Gamma_g(l)\backslash \H_g$ is a smooth complex submanifold. Special subvarities of dimension one are called \emph{Shimura curves}.\\

In this paper we investigate Shimura curves in $A_g$ inside the Prym locus and cuting the locus $\sP^{\circ}_g$ non-trivially. As indicated in the introduction, the Coleman-Oort conjecture asserts that for $g$ large enough, there are no special subvarieties  inside the Prym locus and cuttng the locus $\sP^{\circ}_g$ non-trivially. In this note we investigate this conjecture for subvarieties coming from families of Galois covers of curves. Indeed we consider families whose fibers are $\tilde{C_t}\to \P^1$ which give rise to subvarieties in $M_g$ and, by the Prym map, in $A_g$. Any such variety naturally lies in the Prym locus and intersects $\sP^{\circ}_g$. The families of $\tilde{G}$-coverings of $\P^1$ are \emph{semi-stable} in the sense that the resulting family of Prym varieties are semistable abelian varieties. This can always be achieved by a finite base change.  

\subsection{Higgs bundles}
Let $\overline{f}:\overline{S}\to\overline{B}$ be a family of semi-stable curves giving rise to the family $R(\tilde{G},\tilde{\theta},\sigma)\subset R_g$ and representing a curve $D$ in the Prym locus $\sP_g$. Note that, as it is mentioned in \cite{CLZ}, Remark 3.(iii) after a finite base change if necessary, the $\tilde{G}$-action on the fibers induces a $\tilde{G}$-action on the surface $\overline{S}$ which restricts to the $\tilde{G}$-action on the fibers. It is this action on $\overline{S}$ that we use in this paper to appply the theory of cyclic covers. Let us observe some properties of this family: To this family is associated a Higgs bundle with local system $\V_B:=R^1f_*\Q_{\overline{S}\setminus\Delta}$. This Higgs bundle is $(E^{1,0}_{\overline{B}}\oplus E^{0,1}_{\overline{B}}, \theta_{\overline{B}})$ in which
\[E^{1,0}_{\overline{B}}=\overline{f}_*\Omega_{\overline{S}/\overline{B}}, E^{0,1}_{\overline{B}}=R^1\overline{f}_*\sO_{\overline{S}}\]
and the Higgs field is
\[\theta_{\overline{B}}:E^{1,0}_{\overline{B}}\to E^{0,1}_{\overline{B}}\otimes\Omega_{\overline{B}}(\log\Delta_{nc}) \]
There is a decomposition 
\[(E^{1,0}_{\overline{B}}\oplus E^{0,1}_{\overline{B}}, \theta_{\overline{B}})=(A^{1,0}_{\overline{B}}\oplus A^{0,1}_{\overline{B}}, \theta_{\overline{B}}|_{A^{1,0}_{\overline{B}}})\oplus (F^{1,0}_{\overline{B}}\oplus F^{0,1}_{\overline{B}}, 0)\]
where $A^{1,0}_{\overline{B}}$ is ample and $F^{1,0}_{\overline{B}}\oplus F^{0,1}_{\overline{B}}$ is flat and corresponds to a unitary local subsystem $\V^u_B\subset \V_B\otimes\C$. \\

The action of the involution $\sigma\in\tilde{G}$, gives us a decomposition $\V_B=\V_{B,+}\oplus \V_{B,-}$. The local system $\V_{B,+}$ descends to the quotient family $C_t\to\P^1$ and the local system $\V_{B,-}=R^1f_*\Q_{\overline{S}\setminus\Delta,-}$ is related to the family of Prym varieties. Correspondingly, the associated Higgs bundle is $(E^{1,0}_{\overline{B},-}\oplus E^{0,1}_{\overline{B},-}, \theta_{\overline{B},-})$ in which
\[E^{1,0}_{\overline{B},-}=\overline{f}_*\Omega_{\overline{S}/\overline{B},-}, E^{0,1}_{\overline{B},-}=R^1\overline{f}_*\sO_{\overline{S},-}\]
with Higgs field
\[\theta_{\overline{B},-}:E^{1,0}_{\overline{B},-}\to E^{0,1}_{\overline{B},-}\otimes\Omega_{\overline{B}}(\log\Delta_{nc}) \]
Where $\Delta_{nc}$ here denotes the fibres with non-compact Prym varieties. There is an associated decomposition 
\[(E^{1,0}_{\overline{B},-}\oplus E^{0,1}_{\overline{B},-}, \theta_{\overline{B},-})=(A^{1,0}_{\overline{B},-}\oplus A^{0,1}_{\overline{B},-}, \theta_{\overline{B},-}|_{A^{1,0}_{\overline{B},-}})\oplus (F^{1,0}_{\overline{B},-}\oplus F^{0,1}_{\overline{B},-}, 0)\]
where $A^{1,0}_{\overline{B},-}$ is ample and $F^{1,0}_{\overline{B},-}\oplus F^{0,1}_{\overline{B},-}$ is flat and corresponds to a unitary local subsystem $\V^u_{B,-}\subset \V_{B,-}\otimes\C$.\\

According to the characterization in \cite{VZ04},
\begin{equation}
D \text{ is a Shimura curve }\Leftrightarrow \deg E^{1,0}_{\overline{B},-}=\frac{\rk A^{1,0}_{\overline{B},-}}{2}\deg\Omega^1_{\overline{B}}(\log\Delta_{nc})
\end{equation}
If $D$ is a non-compact Shimura curve,
\begin{equation} \label{sing fib}
p(\overline{F})=\rk F^{1,0}_{\overline{B},-} \hspace{1 cm} \text{ for any fiber } \overline{F}:=\tilde{C}\to C \text{ over } \Delta_{nc},
\end{equation}
where $p(\overline{F})=\dim P(\tilde{C}/C)$. This follows for isntance from \cite{LZ}, Theorem 0.2.\\

Since there is a $\tilde{G}$-action on the surface $\overline{S}$, we have an action of $G=\tilde{G}/\langle \sigma\rangle\subset \tilde{G}$ and the corresponding eigenspace decomposition 
\[\V_{B,-}\otimes\C=\bigoplus_{i=0}^{n-1}\V_{B,-,i};  (E^{1,0}_{\overline{B},-}\oplus E^{0,1}_{\overline{B},-}, \theta_{\overline{B},-})=\bigoplus_{i=0}^{n-1}(E^{1,0}_{\overline{B},-}\oplus E^{0,1}_{\overline{B},-}, \theta_{\overline{B},-})_i\]
and furthermore
\[p=\sum_{i=0}^{n-1}\rk E^{1,0}_{\overline{B},-,i}\]
and also 
\[\rk E^{1,0}_{\overline{B},-,i}=\rk E^{0,1}_{\overline{B},-,n-i}\]

The eigenspace decomposition is compatible with this decomption and we have
\begin{align}
&(A^{1,0}_{\overline{B},-}\oplus A^{0,1}_{\overline{B},-}, \theta_{\overline{B},-})=\bigoplus_{i=0}^{n-1}(A^{1,0}_{\overline{B},-}\oplus A^{0,1}_{\overline{B},-}, \theta_{\overline{B},-}|_{A^{1,0}_{\overline{B},-}})_i,\\
&(F^{1,0}_{\overline{B},-}\oplus F^{0,1}_{\overline{B},-}, 0)=\bigoplus_{i=0}^{n-1}(F^{1,0}_{\overline{B},-}\oplus F^{0,1}_{\overline{B},-}, 0)_i
\end{align}

Each local subsystem $\V_{\overline{B},-,i}$ is defined over the $n$-th cyclotomic field $\mathbb{Q}(\xi_n)$ and so the arithmetic Galois group $\gal(\mathbb{Q}(\xi_n)/\mathbb{Q})$ has a natural action on he above decompositions.
\begin{remark}\label{not cyc}
It is proven in \cite{CFGP}, Lemma 6.1 that if $(\tilde{G}, \theta)$ is an unramified Prym datum, then $\tilde{G}$ is not cyclic. In the following we consider therefore families of $\Z_{2n}$-covers of $\P^1$ in the ramified Prym locus $R_{g,2}$. We prefer to state our results first for the cyclic covers for simplicity. We subsequently consider families of abelian covers of $\P^1$ both in ramified and unramified case. 
\end{remark}
\begin{lemma}\label{triv}
Let $\overline{f}:\overline{S}\to\overline{B}$ be a family of semi-stable $\Z_{2n}$-covers representing a Shimura curve $D$ in the (necessarily ramified) Prym locus $\sP_g$. Let $\V^{tr}_{\overline{B},-,i}\subset\V_{\overline{B},-,i}$ be the trivial local subsystem, and $\Big((F^{1,0}_{\overline{B},-,i})^{tr}\oplus (F^{0,1}_{\overline{B},-,i})^{tr}, 0\Big)$ be the associated trivial flat subbundle. If $\V_{\overline{B},-,i}$ and $\V_{\overline{B},-,j}$ are in one  $\gal(\mathbb{Q}(\xi_{n})/\mathbb{Q})$-orbit, then
\begin{align}\label{trivial loc}
&\rk \V^{tr}_{\overline{B},-,i}=\rk \V^{tr}_{\overline{B},-,j}\\
&\rk (F^{1,0}_{\overline{B},-,i})^{tr}+\rk (F^{1,0}_{\overline{B},-,n-i})^{tr}=\rk (F^{1,0}_{\overline{B},-,j})^{tr}+\rk (F^{1,0}_{\overline{B},-,n-j})^{tr}
\end{align}
In particular, if $n=p$ is a prime number, for any $1\leq i<j\leq p-1$, 
\begin{align}
&\rk \V^{tr}_{\overline{B},i}=\rk \V^{tr}_{\overline{B},j}\\
&\rk (F^{1,0}_{\overline{B},i})^{tr}+\rk (F^{1,0}_{\overline{B},p-i})^{tr}=\rk (F^{1,0}_{\overline{B},j})^{tr}+\rk (F^{1,0}_{\overline{B},p-j})^{tr}
\end{align}
\end{lemma}
\begin{proof}
The equality \ref{trivial loc} follow from the fact that trivial local systems correspond to trivial representations and trivial representations remain trivial under Galois conjugation. Also $\Big((F^{1,0}_{\overline{B},-,i})^{tr}\oplus (F^{0,1}_{\overline{B},-,i})^{tr}, 0\Big)$ is mapped isomorphically to $\Big((F^{1,0}_{\overline{B},-,n-i})^{tr}\oplus (F^{0,1}_{\overline{B},-,n-i})^{tr},0\Big)$ for any $1\leq i\leq n-1$. Moreover under this isomorphism, $(F^{1,0}_{\overline{B},-,i})^{tr}\cong (F^{0,1}_{\overline{B},-,n-i})^{tr}$ and $(F^{0,1}_{\overline{B},-,i})^{tr}\cong (F^{1,0}_{\overline{B},-,n-i})^{tr}$. If $n=p$ is a prime number, the Galois group $\gal(\mathbb{Q}(\xi_{p})/\mathbb{Q})$ acts transitively and permutes the eigenspaces. 
\end{proof}

\begin{remark}\label{abfam}
The families of abelian covers that we consider have abelian group $\tilde{G}=\Z_{2n}\times (\Z_{n})^{m-1}$ (and hence $G=\tilde{G}/\langle\sigma\rangle=(\Z_{n})^{m}$) and satisfy the following condition: The matrix $A$ of the abelian cover has only entries 0,1 and in each column there is exactly one 1 entry. 
\end{remark}

For the abelian covers satisfying the conditions of Remark \ref{abfam}, we can prove the following analogous statement for abelian covers
\begin{lemma}\label{triv}
Let $\overline{f}:\overline{S}\to\overline{B}$ be a family of semi-stable abelian $\tilde{G}$-covers as in Remark \ref{abfam} representing a Shimura curve $D$ in the (ramified or unramified) Prym locus $\sP_g$. Let $\chi$ be a character of $G$ and let $\V^{tr}_{\overline{B},-,\chi}\subset\V_{\overline{B},-,\chi}$ be the trivial local subsystem, and $\Big((F^{1,0}_{\overline{B},-,\chi})^{tr}\oplus (F^{0,1}_{\overline{B},-,\chi})^{tr}, 0\Big)$ be the associated trivial flat subbundle. If $\V_{\overline{B},-,\chi}$ and $\V_{\overline{B},-,\chi^{\prime}}$ are in one  $\gal(\mathbb{Q}(\xi_{n})/\mathbb{Q})$-orbit, then
\begin{align}\label{trivial loc}
&\rk \V^{tr}_{\overline{B},-,\chi}=\rk \V^{tr}_{\overline{B},-,\chi^{\prime}}\\
&\rk (F^{1,0}_{\overline{B},-,\chi})^{tr}+\rk (F^{1,0}_{\overline{B},-,\chi^{-1}})^{tr}=\rk (F^{1,0}_{\overline{B},-,\chi^{\prime}})^{tr}+\rk (F^{1,0}_{\overline{B},-,{\chi^{\prime}}^{-1}})^{tr}
\end{align}
In particular, if $n=p$ then for any two character $\chi, \chi^{\prime}$, 
\begin{align}
&\rk \V^{tr}_{\overline{B},\chi}=\rk \V^{tr}_{\overline{B},\chi^{\prime}}\\
&\rk (F^{1,0}_{\overline{B},\chi})^{tr}+\rk (F^{1,0}_{\overline{B},\chi^{-1}})^{tr}=\rk (F^{1,0}_{\overline{B},\chi^{\prime}})^{tr}+\rk (F^{1,0}_{\overline{B},{\chi^{\prime}}^{-1}})^{tr}
\end{align}
\end{lemma}
The following Lemma gives information about the flat and non-flat part of the Higgs bundle of a family of curves representing Shimura curves in the Prym locus. 

\begin{lemma}\label{higgsrek}
Let $\overline{f}:\overline{S}\to\overline{B}$ be a family of semi-stable $\Z_{2n}$-covers representing a Shimura curve $D$ in the ramified Prym locus $\sP_g$. Then
\begin{equation}\label{rkeqs}
\rk A^{1,0}_{\overline{B},-,i}=\rk A^{0,1}_{\overline{B},-,i}=\rk A^{1,0}_{\overline{B},-,n-i} \forall 1\leq i\leq n-1
\end{equation}
We also have, 
\begin{gather}
\rk F^{1,0}_{\overline{B},-,i}\neq 0 \hspace{2cm}\text{ if        }\rk E^{1,0}_{\overline{B},-,i}>\rk E^{1,0}_{\overline{B},-,n-i};\label{rknzero}\\
\rk F^{1,0}_{\overline{B},-,n-i}\geq \rk F^{1,0}_{\overline{B},-,i} \hspace{2cm}\text{      for        } i\geq n/2 \label{rkgreater}
\end{gather}
\end{lemma}
\begin{proof}
Since $D$ is a Shimura curve, there is a decomposition 
\[(E^{1,0}_{\overline{D},-}\oplus E^{0,1}_{\overline{D},-}, \theta_{\overline{D},-})=(A^{1,0}_{\overline{D}}\oplus A^{0,1}_{\overline{D},-}, \theta_{\overline{D},-}|_{A^{1,0}_{\overline{D},-}})\oplus (F^{1,0}_{\overline{D},-}\oplus F^{0,1}_{\overline{D},-}, 0)\]
in which the Higgs field $\theta_{\overline{D},-}|_{A^{1,0}_{\overline{D},-}}$ is an isomorphism. Note that $(A^{1,0}_{\overline{B},-}\oplus A^{0,1}_{\overline{B},-}, \theta_{\overline{B},-}|_{A^{1,0}_{\overline{B},-}})$ is the pull-back of $(A^{1,0}_{\overline{D},-}\oplus A^{0,1}_{\overline{D},-}, \theta_{\overline{D},-}|_{A^{1,0}_{\overline{D},-}})$ under an isomorphism and hence $\theta_{\overline{B},-}|_{A^{1,0}_{\overline{B},-}}$ must also be an isomorphism. Since the Higgs field respects the eigenspace decomposition, by restricting to $(A^{1,0}_{\overline{B},-}\oplus A^{0,1}_{\overline{B},-}, \theta_{\overline{B},-}|_{A^{1,0}_{\overline{B},-}})_i$ we obtain that $\theta_{\overline{B},-}|_{A^{1,0}_{\overline{B},-,i}}$ is an isomorphism. This proves that 
$\rk A^{1,0}_{\overline{B},-,i}=\rk A^{0,1}_{\overline{B},-,i}$ as claimed. The second equality follows by complex conjugating. Furthermore, \ref{rkeqs} implies that 
\[\rk E^{1,0}_{\overline{B},-,i}-\rk E^{1,0}_{\overline{B},-,n-i}=\rk F^{1,0}_{\overline{B},-,i}-\rk F^{1,0}_{\overline{B},-,n-i}, 1\leq i\leq n-1\]
or equivalently,
\[\rk F^{1,0}_{\overline{B},-,i}=\rk F^{1,0}_{\overline{B},-,n-i}+(\rk E^{1,0}_{\overline{B},-,i}-\rk E^{1,0}_{\overline{B},-,n-i}),  1\leq i\leq n-1\]
from which \ref{rknzero} follows. Finally, for $i\geq n/2$, again using the above equality together with Corollary \ref{cyc tot} and Lemma \ref{eigspacedim} gives
\[ \rk F^{1,0}_{\overline{B},-,n-i}-\rk F^{1,0}_{\overline{B},-,i}=\rk E^{1,0}_{\overline{B},-,n-i}-\rk E^{1,0}_{\overline{B},-,i}=\frac{r(2i-n)}{n}\geq 0\]
which implies \ref{rkgreater}. 
\end{proof}

\begin{lemma}\label{higgsrek ab}
Let $\overline{f}:\overline{S}\to\overline{B}$ be a family of semi-stable $\tilde{G}$-covers as in Remark \ref{abfam} representing a Shimura curve $D$ in the ramified or unramified Prym locus $\sP_g$ or $\sP_{g-1}$. Then
\begin{equation}\label{rkeqs2}
\rk A^{1,0}_{\overline{B},-,\chi}=\rk A^{0,1}_{\overline{B},-,\chi}=\rk A^{1,0}_{\overline{B},-,\chi^{-1}} \forall \chi\in G^*
\end{equation}
If the chacater $\chi$ corresponds to an element $i=(i_1,\ldots, i_m)\in G$ then

\begin{equation}\label{rk1 ab}
\rk F^{1,0}_{\overline{B},-,\chi}\neq 0 \hspace{2cm}\text{ if        }\rk E^{1,0}_{\overline{B},-,\chi}>\rk E^{1,0}_{\overline{B},-,\chi^{-1}};
\end{equation}
If furthermore $n=p$ is a prime number, then
\begin{equation}\label{rk2 ab}
\rk F^{1,0}_{\overline{B},-,\chi^{-1}}\geq \rk F^{1,0}_{\overline{B},-,\chi} \hspace{2cm}\text{      if        } i_j\geq \frac{p}{2} \text{     for   every     } 1\leq j\leq n
\end{equation}
\end{lemma}
\begin{proof}
The proof of \ref{rkeqs2} and \ref{rk1 ab} is similar to the cyclic case. For \ref{rk2 ab} and using the condition satisfied by abelian covers as in Remark \ref{abfam} we see as in the proof of \ref{higgsrek} that
\[ \rk F^{1,0}_{\overline{B},-,\chi^{-1}}-\rk F^{1,0}_{\overline{B},-,\chi}=\rk E^{1,0}_{\overline{B},-,\chi^{-1}}-\rk E^{1,0}_{\overline{B},-,\chi}=\sum\frac{r(2i_j-p)}{2p}\geq 0\]
\end{proof}

\begin{lemma}\label{second fib}
Let $\overline{f}:\overline{S}\to\overline{B}$ be a family of semi-stable $\Z_{2n}$-covers of $\P^1$ representing a non-compact Shimura curve $\sD$ in the ramified Prym locus $\sP_g$. Suppose that after a suitable base change of $\overline{B}$, there exists a fibration $\overline{f}^{\prime}:\overline{S}\to\overline{B}^{\prime}$ such that $g(\overline{B}^{\prime})\geq  \rk F^{1,0}_{\overline{B},-} $. Then $g<8$. 
\end{lemma}
\begin{proof}
Assume that such a fibration $\overline{f}^{\prime}$ with the above properties exists but $\tilde{g}\geq 16$. Then restricting $\overline{f}^{\prime}$ to an arbitrary fiber $\tilde{C}$ of $\overline{f}$ we get a morphism $\overline{f}^{\prime}|_{\tilde{C}}:\tilde{C}\to \overline{B}^{\prime}$. Note that $\deg(\overline{f}^{\prime}|_{\tilde{C}})$ does not depend on choice of the fiber $\tilde{C}$. We claim that $\deg(\overline{f}^{\prime}|_{\tilde{C}})> 2$. Note that since the fibration $\overline{f}$ is non-isotrivial, we have that $\deg(\overline{f}^{\prime}|_{\tilde{C}})>1$. Futhermore, since we have a family of $\Z_{2n}$-covers of $\P^1$, it follows from the equations \ref{abelian equation} that the fibers are not double covers of a fixed curve. This proves our claim. 


The Riemann-Hurwitz theorem then implies that
\begin{equation}\label{RH}
2g(\tilde{C})-2\geq 3(2g(\overline{B}^{\prime})-2)
\end{equation}
This together with our assumption yields
\[2g=g(\tilde{C})\geq 3g(\overline{B}^{\prime})-2\geq 3\rk F^{1,0}_{\overline{B},-}-2\]
Now by Lemma ~\ref{higgsrek} and Corollary ~\ref{cyc tot}, we have
\begin{equation}\label{ineqrk}
\rk F^{1,0}_{\overline{B},-}\geq\displaystyle\sum_{i=1}^{n}\rk F^{1,0}_{\overline{B},-,i}\geq\sum_{i=1}^{n}(\rk E^{1,0}_{\overline{B},-,i}-\rk E^{0,1}_{\overline{B},-,i})=\frac{rn}{8}
\end{equation}
Now by the assumptions of the lemma on $n$, the number in the last equality in \ref{ineqrk} is greater than $2$. Since $\sD$ is non-compact, if in ~\ref{RH} we choose a fiber $\overline{F}$ over $\Delta_{nc}$ we get a contradiction to \ref{sing fib} because $g(\overline{B}^{\prime})\geq  \rk F^{1,0}_{\overline{B},-}\geq 3$. Hence we have $\tilde{g}< 15$, or equivalently, $g<8$.
\end{proof}


In the abelian case we have the following analogous statement

\begin{lemma}\label{second fib ab}
Let $\overline{f}:\overline{S}\to\overline{B}$ be a family of semi-stable $\Z_{2n}\times(\Z_{n})^{m-1}$-covers of $\P^1$ as in Remark \ref{abfam} representing a non-compact Shimura curve $\sD$ in the (ramified or unramified) Prym locus $\sP_g$. Suppose that after a suitable base change of $\overline{B}$, there exists a fibration $\overline{f}^{\prime}:\overline{S}\to\overline{B}^{\prime}$ such that $g(\overline{B}^{\prime})\geq  \rk F^{1,0}_{\overline{B},-} $. Then $g<8$. 
\end{lemma}

\section{Shimura families in the Prym locus} \label{Shimura in Prym}

Let $\overline{f}:\overline{S}\to\overline{B}$ be a family of semi-stable $\Z_{2n}$-covers of $\P^1$. We will assume that the $\Z_{2n}$-action extends to $\overline{S}$ and remark that this can always be achieved using suitable finite covers. The Galois cover $\overline{\Pi}:\overline{S}\to\overline{Y}$ induces a Galois cover $\overline{\Pi}^{\prime}:S^{\prime}\to\widetilde{Y}$ whose branch locus is a normal crossing divisor and with $\gal(\overline{\Pi}^{\prime})\cong G$, where $\widetilde{Y}\to\overline{Y}$ is the minimal resolution of singularities. Let $\widetilde{S}\to S^{\prime}$ be the minimal resolution of
singularities. Then there is an induced birational contraction $\widetilde{S}\to \overline{S}$. Since $\Pi^{\prime}$ is a $\Z_{2n}$-cover of surfaces, it is defined by a relation $\sL^n\equiv\sO_Y(R)$. \\

Restricting to a general fiber $\widetilde{\Gamma}$ of $\widetilde{\varphi}$, one has
\begin{equation}\label{restr}
\mathcal{L}^{(i)}|_{\widetilde{\Gamma}}=\begin{cases}
\sO_{\P^1}(\frac{i\alpha}{2n}), & \text{if } \frac{i\alpha}{2n}\in\Z,\\
\sO_{\P^1}(\lfloor\frac{i\alpha}{2n}\rfloor+1),  & \text{otherwise.} 
\end{cases} 
\end{equation}
Let $\widetilde{R}\subseteq\widetilde{Y}$  be the reduced branch divisor of $\Pi^{\prime}$. Then by \cite{Vie82}, Lemma 1.7 one has the inclusion
\begin{equation}\label{inclu vie}
\tau_1: \widetilde{\Pi}_* \Omega^1_{\widetilde{S}}\hookrightarrow \Omega^1_{\widetilde{Y}}\bigoplus(\bigoplus_{j=1}^{2n-1}\Omega^1_{\widetilde{Y}}(\widetilde{R})\otimes {\mathcal{L}^{(j)}}^{-1})
\end{equation}

\begin{lemma}\label{inclu eigen}
For any $1\leq i\leq n-1$, there is a sheaf morphism 
\begin{equation}\label{sheaf morph}
\rho_i:\widetilde{\varphi}^*F^{1,0}_{\overline{B},-,i}\to \widetilde{\Pi}_*\Omega^1_{\widetilde{S}}
\end{equation}
such that the induced canonical morphism
\begin{equation}\label{sheaf homo}
F^{1,0}_{\overline{B},-,i}=\widetilde{\varphi}_*\widetilde{\varphi}^*F^{1,0}_{\overline{B},-,i}\to \widetilde{\varphi}_*\widetilde{\Pi}_* \Omega_{\widetilde{S}}=\widetilde{f}_*\Omega^1_{\widetilde{S}}=\overline{f}_*\Omega^1_{\overline{S}}\to \overline{f}_*\Omega^1_{\overline{S}/\overline{B}}(\log\Upsilon)=E^{1,0}_{\overline{B}}
\end{equation}
conincides with the inclusion $F^{1,0}_{\overline{B},-,i}\hookrightarrow F^{1,0}_{\overline{B}}\to E^{1,0}_{\overline{B}}$. Moreover, we may choose $\rho_i$ so that the image of $\rho_i$ is contained in $\Omega_{\widetilde{Y}}(\widetilde{R})_-\otimes {\mathcal{L}^{(i)}}^{-1}$ under the inclusion \ref{inclu vie}. 
\end{lemma}
\begin{proof}
According to the proof of \cite{CLZ}, Lemma 4.8, there exists a sheaf morphism 
\begin{equation}\label{sheaf homo3}
\rho: \widetilde{\varphi}^*F^{1,0}_{\overline{B}}\to \widetilde{\Pi}_*\Omega^1_{\widetilde{S}}
\end{equation}
which satisfies the above properties. By restricting first to $\widetilde{\varphi}^*F^{1,0}_{\overline{B},-}$ and subsequently to $\widetilde{\varphi}^*F^{1,0}_{\overline{B},-,i}$ we obtain $\rho_i$ as in \ref{sheaf morph} such that the induced morphism 
\ref{sheaf morph} coincides with the inclusion $F^{1,0}_{\overline{B},-,i}\hookrightarrow F^{1,0}_{\overline{B}}\to E^{1,0}_{\overline{B}}$. Combining \ref{sheaf homo3} with \ref{sheaf homo}, one obtains a sheaf morphism
\[\widetilde{\varphi}^*F^{1,0}_{\overline{B}}\to \Omega^1_{\widetilde{Y}}\bigoplus(\bigoplus_{j=1}^{2n-1}\Omega^1_{\widetilde{Y}}
(\widetilde{R})\otimes {\mathcal{L}^{(j)}}^{-1})\]
The above morphism respects the $\tilde{G}$-action and in particular the $\sigma$-action, we obtain that the image of $\rho_i$ lies in $\Omega_{\widetilde{Y}}(\widetilde{R})_-\otimes {\mathcal{L}^{(i)}}^{-1}$. 
\end{proof}

\begin{lemma}\label{trivial base change}
Let $\widetilde{R}\subseteq\widetilde{Y}$ be the reduced branch divisor of $\Pi^{\prime}$ as above, and $\widetilde{\Gamma}$ be a general fiber of $\widetilde{\varphi}$. If $\widetilde{R}$ contains at least one section of $\widetilde{\varphi}$, and there exist $1\leq i_1\leq i_2\leq n-1$ such that $\rk F^{1,0}_{\overline{B},-, i_1}\neq 0, \rk F^{1,0}_{\overline{B},-, i_2}\neq 0$ and
\begin{equation}\label{inter}
\widetilde{\Gamma}\cdotp(\omega_{\widetilde{Y}}(\widetilde{R})_-\otimes ({\mathcal{L}^{(i_1)}}^{-1}\otimes {\mathcal{L}^{(i_2)}}^{-1}))<0
\end{equation}
Then both $F^{1,0}_{\overline{B},-, i_1}, F^{1,0}_{\overline{B},-, i_2}$ become trivial after a suitable finite \'etale base change.
\end{lemma}

\begin{proof}
This is a generalization of \cite{CLZ}, Lemma 4.9. In the first step, one shows that for any non-zero unitary subbundle $\sU\subseteq F^{1,0}_{\overline{B},-, i}$  with $i=i_1, i_2$, the image of $\rho_i(\widetilde{\varphi}^*\sU)$ is an invertible sheaf $M$ which is nef and satisfies $M^2=0$ and $M\cdotp D=0$ for any component $D\subseteq \widetilde{R}$. By Lemma \ref{inclu eigen}, $\rho_{i_j}(\widetilde{\varphi}^*\sU)$ is contained in $\Omega_{\widetilde{Y}}(\widetilde{R})_-\otimes {\mathcal{L}^{(i)}}^{-1}$ for $j=1,2$ and it is non-zero. In the same way as in \cite{LZ}, Lemma 7.3, it suffices to show that the image $\rho_{i_j}(\widetilde{\varphi}^*\sU)$ is a susheaf of rank one. We prove this by contradiction. If the claim is not true, consider the wedge product
\begin{align}\label{wedge}
\tau\circ\rho_{i_1}\wedge\tau\circ\rho_{i_2}: \widetilde{\varphi}^*\sU\otimes \widetilde{\varphi}^*F^{1,0}_{\overline{B},-,i_2}\to \wedge^2\Omega^1_{\widetilde{Y}}(\widetilde{R})_-\otimes ({\mathcal{L}^{(i_1)}}^{-1}\otimes {\mathcal{L}^{(i_2)}}^{-1})\to \omega_{\widetilde{Y}}(\widetilde{R})\otimes ({\mathcal{L}^{(i_1)}}^{-1}\otimes {\mathcal{L}^{(i_2)}}^{-1})
\end{align}
Denote the image of the above map by $\sC$. We show that $\sC$ is semi-positive in the sense that if $\sE$ is a locally free sheaf on $\widetilde{Y}$, and $\psi:Z\to \widetilde{Y}$ is any morphism from a smooth conplete curve $Z$, the pull-back $\psi^* \sE$ has no quotient line bundle of negative degree. For such a morphism $\psi$, the sheaf $\psi^* (\widetilde{\varphi}^*F^{1,0}_{\overline{B},-,i_1}\otimes\widetilde{\varphi}^*F^{1,0}_{\overline{B},-,i_2})$ is poly-stable of slope zero as it comes from a unitary representation which implies that $\widetilde{\varphi}^*\sU\otimes \widetilde{\varphi}^*F^{1,0}_{\overline{B},-,i_2}$ is semi-positive. This shows that $\sC$ as a quotient of $\widetilde{\varphi}^*\sU\otimes \widetilde{\varphi}^*F^{1,0}_{\overline{B},-,i_2}$ is also semi-positive. Furthermore, it follows from \ref{inter} that $\omega_{\widetilde{Y}}(\widetilde{R})\otimes ({\mathcal{L}^{(i_1)}}^{-1}\otimes {\mathcal{L}^{(i_2)}}^{-1})$ cannot contain any non-zero semi-positive subsheaf. It contradicts the semipositivity of $\sC$. Finally using \cite{De71}, \S 4.2, it suffices to show that $F^{1,0}_{\overline{B},-, i}$ is a direct sum of lines bundles after suitable etale base change for $i=i_1,i_2$. Let $D\subseteq \widetilde{R}$ be a section and let $F^{1,0}_{\overline{B},-, i}=\oplus_j\sU_{ij}$ be the decomposition into irreducible subbundles. By what we have shown at the beginning for the unitary $\sU_{ij}\subseteq F^{1,0}_{\overline{B},-, i}$, we obtain $M_{ij}\cdotp D=0$, i.e., $\deg\sO_D(M_{ij})=0$ where $M_{ij}=\rho_i(\widetilde{\varphi}^*\sU_{ij})$. As $D$ is a section, $D\cong \overline{B}$. Therefore the sheaf $\sO_D(M_{ij})$ can be viewed as an invertible sheaf on $\overline{B}$, which is a quotient of $\sU_{ij}$, as $M_{ij}$ is a quotient of $\widetilde{\varphi}^*\sU_{ij}$. Since $\sU_{ij}$ comes from a unitary local system, $\sU_{ij}$ is poly-stable. Therefore $\sU_{ij}=\sO_D(M_{ij})\oplus\sU^{\prime}_{ij}$. Since $\sU_{ij}$ is irreducible, $\sU_{ij}=\sO_D(M_{ij})$ is a line bundle. 
\end{proof}

\begin{corollary}\label{inter neg1}
Assume also that $\widetilde{R}$ contains at least one section of $\widetilde{\varphi}$ and that $\rk F^{1,0}_{\overline{B},-,i_0}\neq 0$ for some $i_0\geq n/2$. Then after a suitable unramified base change, $F^{1,0}_{\overline{B},-,i}$ becomes trivial for any $n-i_0\leq i\leq  i_0$. 
\begin{proof}
First of all note that $\widetilde{\Gamma}\cdotp\omega_{\widetilde{Y}}(\widetilde{R})_-= \deg \widetilde{R}|_{\widetilde{\Gamma}}$. Therefoere by \ref{restr}, it follows that 
\[\widetilde{\Gamma}\cdotp(\omega_{\widetilde{Y}}(\widetilde{R})_-\otimes ({\mathcal{L}^{(i)}}^{-1}\otimes {\mathcal{L}^{(i_0)}}^{-1}))<0, \forall n-i_0\leq i\leq  i_0.\]
The claim then follows from Lemma ~\ref{trivial base change}. 
\end{proof}
\end{corollary}
The following lemma, follows from the general proof of Lemma 4.12 of \cite{CLZ}. Indeed, all of the steps in the proof of this result respect the action of $\sigma$. 
\begin{lemma}\label{inter neg}
Assume that there exist $1\leq i_1\leq i_2\leq n-1$, such that \ref{inter} holds, and that $H^0(\overline{S}, \Omega^1_{\overline{S}})_{-,i_1}\neq 0$ and $H^0(\overline{S}, \Omega^1_{\overline{S}})_{-,i_2}\neq 0$. Then there exists a unique fibration $\overline{f}^{\prime}:\overline{S}\to \overline{B}^{\prime}$ such that 
\begin{equation}\label{pull-back gen}
H^0(\overline{S}, \Omega^1_{\overline{S}})_{-,i}\subseteq (\overline{f}^{\prime})^* H^0(\overline{B}^{\prime}, \Omega^1_{\overline{B}^{\prime}}) \text{ for } i=i_1, i_2
\end{equation}
\end{lemma}

\begin{corollary}\label{inter neg2}
Assume that $H^0(\overline{S}, \Omega^1_{\overline{S}})_{-,i_0}\neq 0$ for some $i_0\geq n/2$. Then, after a suitable base change, there exists  unique fibration $\overline{f}^{\prime}:\overline{S}\to \overline{B}^{\prime}$ such that 
\begin{equation}\label{pull-back gen2}
\displaystyle\bigoplus_{n-i_0}^{i_0} H^0(\overline{S}, \Omega^1_{\overline{S}})_{-,i}\subseteq (\overline{f}^{\prime})^* H^0(\overline{B}^{\prime}, \Omega^1_{\overline{B}^{\prime}})
\end{equation}
\end{corollary}
\begin{proof}
Since $H^0(\overline{S}, \Omega^1_{\overline{S}})_{-,i_0}\neq 0$, it follows from \ref{inter} that $F^{1,0}_{\overline{B},-, i_0}\neq 0$. Hence by Corollary \ref{inter neg1}, after a suitable base change,  $F^{1,0}_{\overline{B},-, i}$ is trivial for any $n-i_0\leq i\leq i_0$. This together with \ref{rkeqs} and Corollary \ref{cyc tot}, yields
\begin{equation}\label{combin}
\dim H^0(\overline{S}, \Omega^1_{\overline{S}})_{-,n-i_0}=\rk F^{1,0}_{\overline{B},-, n-i_0}\geq \rk F^{1,0}_{\overline{B},-, i_0}=\dim H^0(\overline{S}, \Omega^1_{\overline{S}})_{-,i_0}>0
\end{equation}
By the proof of Corollary \ref{inter neg1}, the assumption \ref{inter} is true for $i_1=i_0$ and $i_2=n-i_0$. Hence by Lemma \ref{inter neg}, there exists a unique fibration $\overline{f}^{\prime}_{n-i_0}:\overline{S}\to \overline{B}^{\prime}_{n-i_0}$ such that
\begin{equation}\label{combin2}
H^0(\overline{S}, \Omega^1_{\overline{S}})_{-,i_0}\oplus H^0(\overline{S}, \Omega^1_{\overline{S}})_{-,n-i_0}\subseteq (\overline{f}^{\prime}_{n-i_0})^*H^0(\overline{B}^{\prime}, \Omega^1_{\overline{B}^{\prime}})
\end{equation}
This holds indeed for any $i$ such that $n-i_0\leq i\leq i_0$ and that $H^0(\overline{S}, \Omega^1_{\overline{S}})_{-,i}\neq 0$, i.e., there exists a unique fibration $\overline{f}^{\prime}_{i}:\overline{S}\to \overline{B}^{\prime}_{n-i_0}$ such that
\begin{equation}\label{combin3}
H^0(\overline{S}, \Omega^1_{\overline{S}})_{-,i_0}\oplus H^0(\overline{S}, \Omega^1_{\overline{S}})_{-,i}\subseteq (\overline{f}^{\prime}_{n-i_0})^*H^0(\overline{B}^{\prime}, \Omega^1_{\overline{B}^{\prime}})
\end{equation}
Since $H^0(\overline{S}, \Omega^1_{\overline{S}})_{-,i_0}\neq 0$, it follows from the uniqueness of $\overline{f}^{\prime}_{i}$'s that these morphisms are all the same. We denote this unique fibration  $\overline{f}^{\prime}:\overline{S}\to \overline{B}^{\prime}$, which is unique and the claim follows. 
\end{proof}

\begin{lemma}\label{second fib1}
Let $\overline{f}:\overline{S}\to\overline{B}$ be a family of semi-stable $\Z_{2n}$-covers of $\P^1$ representing a non-compact Shimura curve $\sD$ in the ramified Prym locus $\sP_g$ and such that $F^{1,0}_{\overline{B},-, i_0}\neq 0$ for some $i_0\geq n/2$. Then there exists a fibration $\overline{f}^{\prime}:\overline{S}\to \overline{B}^{\prime}$ such that $g(\overline{B}^{\prime})\geq \rk F^{1,0}_{\overline{B},-}$. 
\end{lemma}
\begin{proof}
This follows directly from Corollaries \ref{inter neg1} and \ref{inter neg2}. 
\end{proof}

The following concerns the condition $F^{1,0}_{\overline{B},-, i_0}\neq 0$ for families of $\Z_{2p}$-covers of $\P^1$. 

\begin{lemma}\label{rank neq0}
Assume that $p\geq 5$ then 
\begin{equation}\label{non-triviality}
\rk F^{1,0}_{\overline{B},-, i}>0
\end{equation}
for any $\frac{p+1}{2}\leq i<p-1$
\end{lemma}
\begin{proof}
Since this claim is independent of the base change, we may assume by \cite{VZ04}, Corollary 4.4 that the unitary local subsystem $\V_{B,-}^u\subseteq \V_{B,-}\otimes\C$ is trivial after a suitable finite base change, i.e., $\V_{B,-}^u=\V_{B,-}^{tr}$. By the above, we have
\[\rk (F^{1,0}_{\overline{B},-,i})+\rk (F^{0,1}_{\overline{B},-,i})=\rk (F^{1,0}_{\overline{B},-,j})+\rk (F^{0,1}_{\overline{B},-,j}), \forall 1\leq i\leq j\leq p-1\]
One can also check that 
\[\rk (E^{1,0}_{\overline{B},-,i})+\rk (E^{0,1}_{\overline{B},-,i})=\rk (E^{1,0}_{\overline{B},-,j})+\rk (E^{0,1}_{\overline{B},-,j}), \forall 1\leq i\leq j\leq p-1\]
Now this together with \ref{rkeqs} yields
\begin{equation}
\rk A^{1,0}_{\overline{B},-,i}=\rk A^{0,1}_{\overline{B},-,j}, \forall 1\leq i\leq j\leq p-1
\end{equation}
Hence 
\begin{align*}
&\rk F^{1,0}_{\overline{B},-,i}=\rk E^{1,0}_{\overline{B},-,i}-\rk A^{1,0}_{\overline{B},-,i}\\
&= \rk E^{1,0}_{\overline{B},-,i}-\rk A^{1,0}_{\overline{B},-,p-1}= \rk E^{1,0}_{\overline{B},-,i}-(\rk E^{1,0}_{\overline{B},-,p-1}-\rk F^{1,0}_{\overline{B},p-1})\\ 
&\geq \rk E^{1,0}_{\overline{B},-,i}-\rk E^{1,0}_{\overline{B},-,p-1}=\frac{r(p-(i+1))}{p}>0
\end{align*}
The last inequality holds for the chosen $i$ and for $p\geq 5$.
\end{proof}

\begin{theorem}\label{non-comp}
There does not exist any non-compact Shimura curve contained generically in the ramified Prym locus of $2p$-cyclic covers of $\P^1$ in $A_g$ with $g\geq 8$.
\end{theorem}
\begin{proof}
Assume on the contrary that there exists a Shimura curve $D$ in the Prym locus in $A_g$ with $g\geq 8$ or equivalently $\tilde{g}\geq 16$. Let $\overline{f}:\overline{S}\to\overline{B}$ be a family of semi-stable $2p$-cyclic covers of $\P^1$ which represents $D$. Possibly after a base change we may assume that the group $\tilde{G}\cong \Z_{2p}$ and hence $G\cong \Z_{p}$ on $\overline{S}$ and therefore also on the Higgs bundle $(E^{1,0}_{\overline{B},-}\oplus E^{0,1}_{\overline{B},-}, \theta_{\overline{B},-})$ and its subbundles with eigenspace decompositions. Since $D$ is non-compact, by \cite{VZ04}, Corollary 4.4 the flat subbundle $F^{1,0}_{\overline{B},-}$ becomes trivial after a suitable base change, i.e., $F^{1,0}_{\overline{B},-}\cong\sO^{\oplus t}_{\overline{B}}$, where $t=\rk F^{1,0}_{\overline{B},-}$. Furthermore, by Lemma~\ref{second fib1} and Lemma \ref{rank neq0} there exists another fibration $\overline{f}^{\prime}:\overline{S}\to\overline{B}^{\prime}$ different from $\overline{f}$ such that $g(\overline{B}^{\prime})\geq \rk F^{1,0}_{\overline{B},-}$. But this contradicts Lemma \ref{second fib} which proves our claim that $g< 8$. 
\end{proof}

We remark that in paper \cite{CFGP}, Table 2, there are examples (examples 1, 3-4 and 6-7) of families of cyclic covers of the above form (more precisely, $\Z_{6}, \Z_{10}$ and $\Z_{14}$ covers) in the Prym loci $\sP_g$ with $g< 8$.

\subsection{Results in the abelian case}
In this subsection we prove some results as in the preceding section about abelian covers of $\P^1$. In this case, the associated Prym varieties can be in unramified as well as the ramified Prym locus. For simplicity, we denote both loci by $\sP_g$. Let us first formulate the analogous statements as in the last section. Recall from Remark \ref{abfam} that the abelian covers that we consider are totally ramified in the sense that the columns of the matrix $A$ of the covering contains exactly one 1 entry and all other entries equal to 0.\\

Let $\overline{f}:\overline{S}\to\overline{B}$ be a family of semi-stable abelian $\tilde{G}$-covers of $\P^1$. We will assume that the $\tilde{G}$-action extends to $\overline{S}$ and remark that this can always be achieved using suitable finite covers. The Galois cover $\overline{\Pi}:\overline{S}\to\overline{Y}$ induces a Galois cover $\overline{\Pi}^{\prime}:S^{\prime}\to\widetilde{Y}$ whose branch locus is a normal crossing divisor and with $\gal(\overline{\Pi}^{\prime})\cong G$, where $\widetilde{Y}\to\overline{Y}$ is the minimal resolution of singularities. Let $\widetilde{S}\to S^{\prime}$ be the minimal resolution of singularities. Then there is an induced birational contraction $\widetilde{S}\to \overline{S}$. Since $\Pi^{\prime}$ is a $\tilde{G}$-cover of surfaces, it is defined by line bundles $\sL_{\chi}$. \\

Restricting to a general fiber $\widetilde{\Gamma}$ of $\widetilde{\varphi}$, one has
\begin{equation}\label{restr}
\mathcal{L}_{\chi}|_{\widetilde{\Gamma}}=\mathcal{O}_{\mathbb{P}^{1}}(\sum_{1}^{s}
\langle\frac{\alpha_{j}}{N}\rangle)
\end{equation}
Let $\widetilde{R}\subseteq\widetilde{Y}$  be the reduced branch divisor of $\Pi^{\prime}$. Then by \cite{Vie82}, Lemma 1.7 one has the inclusion
\begin{equation}\label{inclu vie}
\tau_1: \widetilde{\Pi}_* \Omega^1_{\widetilde{S}}\hookrightarrow \Omega^1_{\widetilde{Y}}\bigoplus(\bigoplus_{\chi\in G^*}\Omega^1_{\widetilde{Y}}(\widetilde{R})\otimes \mathcal{L}_{\chi})
\end{equation}

Using the conditions of Remark \ref{abfam}, one can prove the following analogous results as in the previous subsection for abelian covers.
\begin{lemma}\label{inclu eigen ab}
For any $\chi\in G^*$, there is a sheaf morphism 
\begin{equation}\label{sheaf morph ab}
\rho_{\chi}:\widetilde{\varphi}^*F^{1,0}_{\overline{B},-,\chi}\to \widetilde{\Pi}_*\Omega^1_{\widetilde{S}}
\end{equation}
such that the induced canonical morphism
\begin{equation}\label{sheaf homo ab}
F^{1,0}_{\overline{B},-,\chi}=\widetilde{\varphi}_*\widetilde{\varphi}^*F^{1,0}_{\overline{B},-,\chi}\to \widetilde{\varphi}_*\widetilde{\Pi}_* \Omega_{\widetilde{S}}=\widetilde{f}_*\Omega^1_{\widetilde{S}}=\overline{f}_*\Omega^1_{\overline{S}}\to \overline{f}_*\Omega^1_{\overline{S}/\overline{B}}(\log\Upsilon)=E^{1,0}_{\overline{B}}
\end{equation}
conincides with the inclusion $F^{1,0}_{\overline{B},-,\chi}\hookrightarrow F^{1,0}_{\overline{B}}\to E^{1,0}_{\overline{B}}$. Moreover, we may choose $\rho_{\chi}$ so that the image of $\rho_{\chi}$ is contained in $\Omega_{\widetilde{Y}}(\widetilde{R})_-\otimes {\mathcal{L}_{\chi}}^{-1}$ under the inclusion \ref{inclu vie}. 
\end{lemma}

\begin{lemma}\label{trivial base change ab}
Let $\widetilde{R}\subseteq\widetilde{Y}$ be the reduced branch divisor of $\Pi^{\prime}$ as above, and $\widetilde{\Gamma}$ be a general fiber of $\widetilde{\varphi}$. If $\widetilde{R}$ contains at least one section of $\widetilde{\varphi}$, and there exist characters $\chi, \chi^{\prime}$ such that $\rk F^{1,0}_{\overline{B},-, \chi}\neq 0, \rk F^{1,0}_{\overline{B},-, \chi^{\prime}}\neq 0$ and
\begin{equation}\label{inter ab}
\widetilde{\Gamma}\cdotp(\omega_{\widetilde{Y}}(\widetilde{R})_-\otimes ({\mathcal{L}_{\chi}}^{-1}\otimes {\mathcal{L}_{\chi^{\prime}}}^{-1}))<0
\end{equation}
Then both $F^{1,0}_{\overline{B},-, \chi}, F^{1,0}_{\overline{B},-, \chi^{\prime}}$ become trivial after a suitable finite \'etale base change.
\end{lemma}

\begin{corollary}\label{inter neg1 ab}
Assume also that $\widetilde{R}$ contains at least one section of $\widetilde{\varphi}$ and that $\rk F^{1,0}_{\overline{B},-,\chi_0}\neq 0$ for a character $\chi_0$ corresponding to an element $i=(i_1,\ldots, i_m)\in G\simeq  (\Z_{p})^{m}$ such that $i_k\geq p/2$. Then after a suitable unramified base change, $F^{1,0}_{\overline{B},-,\chi}$ becomes trivial for all characters $\chi$ corresponding to elements $j=(j_1,\ldots, j_m)\in G$ with $p-i_k\leq j_k\leq  i_k$.
\end{corollary}
The following lemma, follows from the general proof of Lemma 4.12 of \cite{CLZ}. Indeed, all of the steps in the proof of this result respect the action of $\sigma$. 
\begin{lemma}\label{inter neg ab}
Assume that there exist characters $\chi, \chi^{\prime}$, such that \ref{inter ab} holds, and that $H^0(\overline{S}, \Omega^1_{\overline{S}})_{-,\chi}\neq 0$ and $H^0(\overline{S}, \Omega^1_{\overline{S}})_{-,\chi^{\prime}}\neq 0$. Then there exists a unique fibration $\overline{f}^{\prime}:\overline{S}\to \overline{B}^{\prime}$ such that 
\begin{equation}\label{pull-back gen}
H^0(\overline{S}, \Omega^1_{\overline{S}})_{-,\eta}\subseteq (\overline{f}^{\prime})^* H^0(\overline{B}^{\prime}, \Omega^1_{\overline{B}^{\prime}}) \text{ for } \eta=\chi, \chi^{\prime}
\end{equation}
\end{lemma}

\begin{corollary}\label{inter neg2 ab}
Assume that $H^0(\overline{S}, \Omega^1_{\overline{S}})_{-,\chi_0}\neq 0$ for a character $\chi_0$ corresponding to an element $i=(i_1,\ldots, i_n)\in G\simeq  (\Z_{p})^{m}$ such that $i_k\geq p/2$. Then, after a suitable base change, there exists  unique fibration $\overline{f}^{\prime}:\overline{S}\to \overline{B}^{\prime}$ such that 
\begin{equation}\label{pull-back gen2 ab}
\displaystyle\bigoplus_{\chi} H^0(\overline{S}, \Omega^1_{\overline{S}})_{-,\chi}\subseteq (\overline{f}^{\prime})^* H^0(\overline{B}^{\prime}, \Omega^1_{\overline{B}^{\prime}})
\end{equation}
where the sum runs over all characters $\chi$ corresponding to elements $j=(j_1,\ldots, j_m)\in G$ such that $p-i_k\leq j_k\leq i_k$. 
\end{corollary}

\begin{lemma}\label{second fib1 ab}
Let $\overline{f}:\overline{S}\to\overline{B}$ be a semi-stable family of $\Z_{2p}\times (\Z_{p})^{m-1}$-abelian covers of $\P^1$ representing a non-compact Shimura curve $\sD$ in the ramified Prym locus $\sP_g$ and such that $F^{1,0}_{\overline{B},-, \chi_0}\neq 0$ for a character $\chi_0$ corresponding to an element $i=(i_1,\ldots, i_m)\in G\simeq  (\Z_{p})^{m}$ such that $i_k\geq p/2$.. Then there exists a fibration $\overline{f}^{\prime}:\overline{S}\to \overline{B}^{\prime}$ such that $g(\overline{B}^{\prime})\geq \rk F^{1,0}_{\overline{B},-}$. 
\end{lemma}
\begin{proof}
This follows directly from Corollaries \ref{inter neg1 ab} and \ref{inter neg2 ab}. 
\end{proof}

\begin{lemma}\label{rank neq0 ab}
Assume that $p\geq 5$ then 
\begin{equation}\label{non-triviality ab}
\rk F^{1,0}_{\overline{B},-, \chi}>0
\end{equation}
where $\chi$ is the chacarter corresponding to an element $i=(i_1,\ldots, i_m)\in G$ and such that $\frac{p+1}{2}\leq i_j<p-1$ for every $1\leq j\leq m$.
\end{lemma}
\begin{proof}
Note that by our assumption on abelian covers, the group $G=(\Z_{p})^{m}$ acts transitively and permutes the eigenspaces and hence Lemma \ref{triv} holds also in this case. The rest of the proof is as in \ref{rank neq0} for cyclic covers. 
\end{proof}

\begin{theorem}\label{non-comp ab}
There does not exist any non-compact Shimura curve contained generically in the Prym locus of $\tilde{G}=\Z_{2p}\times (\Z_{p})^{m-1}$-abelian covers of $\P^1$ in $A_g$ with $g\geq 8$.
\end{theorem}
\begin{proof}
Assume on the contrary that there exists a Shimura curve $D$ in the Prym locus in $A_g$ with $g\geq 8$ or equivalently $\tilde{g}\geq 16$. Let $\overline{f}:\overline{S}\to\overline{B}$ be a family of semi-stable $\Z_{2p}\times (\Z_{p})^{m-1}$-covers of $\P^1$ which represents $D$. Possibly after a base change we may assume that the group $\tilde{G}$ on $\overline{S}$ and hence $G$ also act on the Higgs bundle $(E^{1,0}_{\overline{B},-}\oplus E^{0,1}_{\overline{B},-}, \theta_{\overline{B},-})$ and its subbundles with eigenspace decompositions. Since $D$ is non-compact, by \cite{VZ04}, Corollary 4.4 the flat subbundle $F^{1,0}_{\overline{B},-}$ becomes trivial after a suitable base change, i.e., $F^{1,0}_{\overline{B},-}\cong\sO^{\oplus t}_{\overline{B}}$, where $t=\rk F^{1,0}_{\overline{B},-}$. Furthermore, by Lemma~\ref{second fib1 ab} and Lemma \ref{rank neq0 ab} there exists another fibration $\overline{f}^{\prime}:\overline{S}\to\overline{B}^{\prime}$ different from $\overline{f}$ such that $g(\overline{B}^{\prime})\geq \rk F^{1,0}_{\overline{B},-}$. But this contradicts Lemma \ref{second fib ab} which proves our claim that $g< 8$. 
\end{proof}

\end{document}